\documentclass{amsart}
\usepackage{amsfonts}

\newtheorem{theorem}{Theorem}
\theoremstyle{plain}

\numberwithin{equation}{section}
\input{tcilatex}

\begin{document}
\title[\textbf{Ruled Weingarten Surfaces related to dual spherical curves}]{%
\textbf{Ruled Weingarten Surfaces related to dual spherical curves}}
\author{\.{I}lkay Arslan G\"{u}ven}
\address{ }
\email{ }
\urladdr{}
\thanks{}
\author{Semra Kaya Nurkan}
\curraddr{ }
\email{ }
\urladdr{}
\thanks{}
\author{Murat Kemal Karacan}
\address{ }
\urladdr{}
\date{}
\dedicatory{}
\thanks{}

\begin{abstract}
We study ruled surfaces in $\mathbb{R}^{3}$ which are obtained from dual
spherical indicatrix curves of dual Frenet vector fields. We find the
Gaussian and mean curvatures of the ruled surfaces and give some results of
being Weingarten surface.
\end{abstract}

\maketitle

\section{Introduction}

E. Study established a relationship of directed lines to dual unit vectors
and he defined a mapping which is called Study mapping. This mapping exists
one-to-one correspondence between the dual points of a dual unit sphere in $%
\mathbb{D}$-Module and the directed lines in $\mathbb{R}^{3}$.

A differentiable curve on the dual unit sphere, depending on a real
parameter $s$, represents a differentiable family of straight lines in $%
\mathbb{R}^{3}$ which is called ruled surface. This correspondence allows us
to study the properties of a ruled surface on the geometry of dual spherical
curves on a dual unit sphere \cite{Guven}.

Ruled surfaces of Weingarten type which have a nontrivial relation holds
between the Gaussian curvature and mean curvature, were studied by many
scientists in \cite{Dillen1, Kuhnel, Sod}. Also properties and inventions of
ruled linear Weingarten surfaces in Minkowski 3-space were investigated \cite%
{Abdel, Dillen1, Dillen2, Dillen3, Sod}.

In this paper we take a unit speed curve on the dual unit sphere and move
each vector of its Frenet frame to the center of dual unit sphere. These
Frenet vectors generates spherical representation curves on the dual unit
sphere. So we have ruled surfaces in $\mathbb{R}^{3}$ which have been
corresponded to representation curves , by Study mapping. Each ruled surface
is determined by a parametrization%
\begin{equation*}
\varphi (s,v)=\alpha (s)+v\overrightarrow{X}(s)
\end{equation*}%
where $\alpha $ is the base curve and $\overrightarrow{X}$ is director
vector field which will be Frenet vector field of dual curve in this paper.

We study on these ruled surfaces that which one is Weingarten and minimal
surface. The conditions of being Weingarten and minimal ruled surfaces were
given by theorems.

\section{Preliminaries}

Dual numbers were defined \ in the 19$^{th}$ century by W.K. Clifford. The
set of all dual numbers $\mathbb{D}$ consists of elements in the form of $%
A=a+\varepsilon a^{\ast }$, where $a$ and $a^{\ast }$ are real numbers, $%
\varepsilon $ is the dual unit with the property of $\varepsilon ^{2}=0$.
The set 
\begin{equation*}
\mathbb{D}=\left\{ A=a+\varepsilon a^{\ast }\mid a,a^{\ast }\epsilon \text{ }%
\mathbb{R}\right\}
\end{equation*}%
forms a commutative ring with the following operations%
\begin{eqnarray*}
i)\text{ \ }(a+\varepsilon a^{\ast })+(b+\varepsilon b^{\ast })
&=&(a+b)+\varepsilon (a^{\ast }+b^{\ast }) \\
ii)\text{ \ }(a+\varepsilon a^{\ast }).(b+\varepsilon b^{\ast })
&=&a.b+\varepsilon (ab^{\ast }+ba^{\ast }).
\end{eqnarray*}

The division of two dual numbers $A=a+\varepsilon a^{\ast }$ and $%
B=b+\varepsilon b^{\ast }$ provided $b\neq 0$ can be defined as%
\begin{equation*}
\frac{A}{B}=\frac{a+\varepsilon a^{\ast }}{b+\varepsilon b^{\ast }}=\frac{a}{%
b}+\varepsilon \frac{a^{\ast }b-ab^{\ast }}{b^{2}}.
\end{equation*}%
The set $\mathbb{D}^{3}$ is a module on the ring $\mathbb{D}$ which is
called $\mathbb{D}$-Module or dual space and is denoted by 
\begin{equation*}
\mathbb{D}^{3}=\mathbb{D}\times \mathbb{D}\times \mathbb{D}=\left\{ 
\begin{array}{c}
\overrightarrow{A}\mid \overrightarrow{A}=\left( a_{1}+\varepsilon
a_{1}^{\ast },a_{2}+\varepsilon a_{2}^{\ast },a_{3}+\varepsilon a_{3}^{\ast
}\right) \\ 
=\left( a_{1}+a_{2}+a_{3}\right) +\varepsilon \left( a_{1}^{\ast
}+a_{2}^{\ast }+a_{3}^{\ast }\right) \\ 
=a+\varepsilon a^{\ast },\text{ \ \ }a\epsilon \mathbb{R}^{3},a^{\ast
}\epsilon \mathbb{R}^{3}%
\end{array}%
\right\}
\end{equation*}%
where the elements are dual vectors. For $a\neq 0$, the norm $\left\Vert 
\overrightarrow{A}\right\Vert $of $\overrightarrow{A}$ is defined by%
\begin{equation*}
\left\Vert \overrightarrow{A}\right\Vert =\sqrt{\left\langle \overrightarrow{%
A},\overrightarrow{A}\right\rangle }=\left\Vert \overrightarrow{a}%
\right\Vert +\varepsilon \frac{\left\langle \overrightarrow{a},%
\overrightarrow{a^{\ast }}\right\rangle }{\left\Vert \overrightarrow{a}%
\right\Vert }.
\end{equation*}%
Now let us give basic concepts of a dual curve and dual Frenet frame.

Let%
\begin{equation*}
\begin{array}{ccc}
\widehat{\alpha }:I & \longrightarrow & \mathbb{D}^{3}\phantom{= V_{i}(s);
\quad 1\leq i\leq n-1} \\ 
\phantom{\alpha : }s & \longrightarrow & \overrightarrow{\widehat{\alpha }}%
(s)=\overrightarrow{\alpha }(s)+\varepsilon \overrightarrow{\alpha ^{\ast }}%
(s)%
\end{array}%
\end{equation*}%
be a dual curve with arc-length parameter $s$. Then 
\begin{equation*}
\frac{d\overrightarrow{\widehat{\alpha }}}{d\widehat{s}}=\frac{d%
\overrightarrow{\widehat{\alpha }}}{ds}.\frac{ds}{d\widehat{s}}=%
\overrightarrow{\widehat{T}}
\end{equation*}%
is called the unit tangent vector of $\widehat{\alpha }(s)$. The derivative
of $\overrightarrow{\widehat{T}}$ is 
\begin{equation*}
\frac{d\overrightarrow{\widehat{T}}}{d\widehat{s}}=\frac{d\overrightarrow{%
\widehat{T}}}{ds}.\frac{ds}{d\widehat{s}}=\frac{d^{2}\overrightarrow{%
\widehat{\alpha }}}{d\widehat{s}^{2}}=\widehat{\kappa }\overrightarrow{%
\widehat{N}}
\end{equation*}%
and the norm of the vector$\frac{d\overrightarrow{\widehat{T}}}{d\widehat{s}}
$ is called curvature function of $\overrightarrow{\widehat{\alpha }}(s)$.
Here $\widehat{\kappa }:I\longrightarrow \mathbb{D}$ is never pure-dual.
Then the unit principal normal vector of $\overrightarrow{\widehat{\alpha }}%
(s)$ is defined as%
\begin{equation*}
\overrightarrow{\widehat{N}}=\frac{1}{\widehat{\kappa }}.\frac{d%
\overrightarrow{\widehat{T}}}{d\widehat{s}}
\end{equation*}%
The vector $\overrightarrow{\widehat{B}}=\overrightarrow{\widehat{T}}\times 
\overrightarrow{\widehat{N}}$ is called the binormal vector of $%
\overrightarrow{\widehat{\alpha }}(s)$. Also we call the vectors $%
\overrightarrow{\widehat{T}},\overrightarrow{\widehat{N}},\overrightarrow{%
\widehat{B}}$ dual Frenet trihedron of $\overrightarrow{\widehat{\alpha }}%
(s) $ at the point $\widehat{\alpha }(s)$. The derivatives of dual Frenet
vectors $\overrightarrow{\widehat{T}},\overrightarrow{\widehat{N}},%
\overrightarrow{\widehat{B}}$ can be written in matrix form as 
\begin{equation*}
\left[ 
\begin{array}{l}
\overrightarrow{\widehat{T}}^{\prime } \\ 
\overrightarrow{\widehat{N}}^{\prime } \\ 
\overrightarrow{\widehat{B}}^{\prime }%
\end{array}%
\right] =\left[ 
\begin{array}{lll}
0 & \widehat{\kappa } & 0 \\ 
-\widehat{\kappa } & 0 & \widehat{\tau } \\ 
0 & -\widehat{\tau } & 0%
\end{array}%
\right] \left[ 
\begin{array}{l}
\overrightarrow{\widehat{T}} \\ 
\overrightarrow{\widehat{N}} \\ 
\overrightarrow{\widehat{B}}%
\end{array}%
\right]
\end{equation*}%
which are called Frenet formulas \cite{Kose}. The function $\widehat{\tau }%
:I\longrightarrow \mathbb{D}$ such that $\frac{d\overrightarrow{\widehat{B}}%
}{d\widehat{s}}=-\widehat{\tau }\overrightarrow{\widehat{N}}$ is called the
torsion of $\overrightarrow{\widehat{\alpha }}(s)$.

A ruled surface in $\mathbb{R}^{3}$ is swept up by a straight line $\ell $
which is moving along a curve $\alpha .$ It is defined by the
parametrization $\varphi (s,v)=\alpha (s)+v\overrightarrow{X}(s)$, where $%
\alpha $ is differentiable base curve and $\overrightarrow{X}$ is a nowhere
vanishing director vector field of $\ell $. The lines $\ell $ are called the
rullings of the surface. A ruled surface is said to be developable if the
Gaussian curvature of surface $K$ is zero. The tangent plane of developable
surface is constant along a fixed ruling. If the mean curvature of surface $%
H $ is zero, then the ruled surface is called minimal surface.

A Weingarten surface is a surface for which the Gaussian curvature $K$ and
the mean curvature $H$ satisfy a nontrivial relation $\Phi (H,K)=0$. For
ruled surfaces in $\mathbb{E}^{3}$ Dini and Beltrami expressed a theorem in
1865 which says that any non-developable ruled Weingarten surface in
Euclidean 3-space $\mathbb{E}^{3}$ is a piece of a helicoidal ruled surface,
defined as the orbit of a straight line under the action of a 1-parameter
group of screw motions. In particular, the Gaussian curvature is nowhere
zero if it is nonzero at some point. The only minimal ruled surface \ is the
classical right helicoid. In \cite{Kuhnel}, this theorem is stated directly
as; among the ruled surfaces, the class of Weingarten surfaces is the set of
all developable surfaces and all helicoidal ruled surfaces.

\section{Gaussian and Mean Curvatures of Ruled Surfaces}

In this section we will compute the Gaussian and mean curvatures of ruled
surfaces which are obtained from dual spherical indicatrix curves of dual
Frenet vector fields. According to the theorem in \cite{Kuhnel}, we will say
that the developable surfaces are Weingarten surfaces. Also we know that if
Gaussian and mean curvature of the ruled surface are zero, then the ruled
surface is called developable and minimal surface, respectively. So we will
investigate the conditions of Gaussian and mean curvatures to being zero.

A curve $\widehat{\alpha }$ is taken on dual unit sphere and dual spherical
representation curves are formed on dual unit sphere by moving\ the Frenet
vectors to the center of dual unit sphere. These representation curves which
we will denote as $(X)$ are corresponded to ruled surfaces in $\mathbb{R}%
^{3} $. The Frenet vectors of $\widehat{\alpha }$ ; unit tangent vector,
unit principal normal vector and unit binormal vector are $\widehat{T}=T+$ $%
\varepsilon T^{\ast }$, \ $\widehat{N}=N+$ $\varepsilon N^{\ast }$ , $%
\widehat{B}=B+$ $\varepsilon B^{\ast }$ , respectively. Also curvature and
torsion of $\widehat{\alpha }$ will be denoted as $\widehat{\kappa }=\kappa
+\varepsilon \kappa ^{\ast }$ and $\widehat{\tau }=\tau +\varepsilon \tau
^{\ast }$. We get the equations below from Frenet formulas and properties of
Frenet vectors:

\begin{eqnarray*}
T\times N &=&B\text{ \ \ \ \ \ \ \ \ \ \ and \ \ \ \ \ \ \ \ \ }T^{\ast
}=N^{\ast }\times B+N\times B^{\ast } \\
N\times B &=&T\text{ \ \ \ \ \ \ \ \ \ \ \ \ \ \ \ \ \ \ \ \ \ \ \ \ \ \ }%
N^{\ast }=B^{\ast }\times T+B\times T^{\ast } \\
B\times T &=&N\text{ \ \ \ \ \ \ \ \ \ \ \ \ \ \ \ \ \ \ \ \ \ \ \ \ \ \ }%
B^{\ast }=T^{\ast }\times N+T\times N^{\ast }
\end{eqnarray*}%
and the derivatives are%
\begin{eqnarray*}
T^{\prime } &=&\kappa N\text{ \ \ \ \ \ \ \ \ \ \ \ \ \ \ \ \ \ \ \ and \ \
\ \ \ \ \ \ }T^{\ast \prime }=\kappa N^{\ast }+\kappa ^{\ast }N \\
N^{\prime } &=&-\kappa T+\tau B\text{ \ \ \ \ \ \ \ \ \ \ \ \ \ \ \ \ \ \ \
\ \ \ \ \ }N^{\ast \prime }=-\kappa T^{\ast }-\kappa ^{\ast }T+\tau B^{\ast
}+\tau ^{\ast }B \\
B^{\prime } &=&-\tau N\text{ \ \ \ \ \ \ \ \ \ \ \ \ \ \ \ \ \ \ \ \ \ \ \ \
\ \ \ \ \ \ \ \ }B^{\ast \prime }=-\tau N^{\ast }-\tau ^{\ast }N.
\end{eqnarray*}%
Now we will express the dual vectoral equations of each ruled surface. We
denote ruled surfaces which are corressponded to dual spherical
representation curves as $\phi _{X}(s,v)$. The equations are 
\begin{eqnarray*}
\phi _{T}(s,v) &=&T(s)\times T^{\ast }(s)+vT(s)\text{ \ \ , \ \ \ \ \ \ }%
T^{\ast }(s)=\beta _{T}(s)\times T(s) \\
\phi _{N}(s,v) &=&N(s)\times N^{\ast }(s)+vN(s)\text{ \ \ , \ \ \ \ \ }%
N^{\ast }(s)=\beta _{N}(s)\times N(s) \\
\phi _{B}(s,v) &=&B(s)\times B^{\ast }(s)+vB(s)\text{ \ \ , \ \ \ \ \ \ }%
B^{\ast }(s)=\beta _{B}(s)\times B(s)
\end{eqnarray*}%
where $\beta _{X}(s)$ is the base curve of ruled surface and $X^{\ast
}(s)=\beta _{X}(s)\times X(s)$ is vectoral moment of the vector $X$ , also $%
s $ is not arc-length parameter of base curve.

Let find the Gaussian curvatures of each ruled surface. We will calculate
Gaussian curvatures from the matrix of shape operator of ruled surfaces.

We take up first ruled surface whose dual vectoral equation is 
\begin{equation*}
\phi _{T}(s,v)=T(s)\times T^{\ast }(s)+vT(s)\text{ \ \ , \ \ }T^{\ast
}(s)=\beta _{T}(s)\times T(s).
\end{equation*}%
The director vector of this ruled surfaces is $T$ and let the normal vector
be $N_{T}$. Since the derivatives%
\begin{eqnarray*}
(\phi _{T})_{v} &=&T \\
(\phi _{T})_{s} &=&-\kappa \langle \beta _{T},N\rangle T-\kappa \langle
\beta _{N},T\rangle N+\kappa ^{\ast }B+v\kappa N
\end{eqnarray*}%
are not orthogonal, we will convert these vectors to a orthonormal base by
Gramm Schmidt method and we will compute the matrix of shape operator with
respect to this orthonormal base.

For \ $(\phi _{T})_{v}=X_{1}$ \ and $\ (\phi _{T})_{s}=X_{2}$ ; the vectors
of orthonormal base $\left\{ E_{1},E_{2}\right\} $ are%
\begin{eqnarray*}
E_{1} &=&\frac{Y_{1}}{\left\Vert Y_{1}\right\Vert }=T \\
E_{2} &=&\frac{Y_{2}}{\left\Vert Y_{2}\right\Vert }=\frac{\kappa (-\langle
\beta _{N},T\rangle +v)N+\kappa ^{\ast }B}{\sqrt{\kappa ^{2}(-\langle \beta
_{N},T\rangle +v)^{2}+\kappa ^{\ast 2}}}
\end{eqnarray*}%
where $Y_{1}=X_{1}=T$ and $Y_{2}=\kappa (-\langle \beta _{N},T\rangle
+v)N+\kappa ^{\ast }B$. By doing simple operation we get $\langle
E_{1},E_{2}\rangle =0$.

The normal vector of this ruled surface is given by%
\begin{eqnarray*}
N_{T} &=&E_{1}\times E_{2} \\
&=&\frac{\kappa (-\langle \beta _{N},T\rangle +v)B-\kappa ^{\ast }N}{\sqrt{%
\kappa ^{2}(-\langle \beta _{N},T\rangle +v)^{2}+\kappa ^{\ast 2}}}.
\end{eqnarray*}

If the matrix of shape operator is $S_{T}$, then 
\begin{equation*}
S_{T}=\left[ 
\begin{array}{cc}
\langle S_{T}(E_{1}),E_{1}\rangle & \langle S_{T}(E_{1}),E_{2}\rangle \\ 
\langle S_{T}(E_{2}),E_{1}\rangle & \langle S_{T}(E_{2}),E_{2}\rangle%
\end{array}%
\right]
\end{equation*}%
can be written. Since $T$ is director vector, it is an asymtotic line so $%
\langle S_{T}(E_{1}),E_{1}\rangle $ $=\langle S_{T}(T),T\rangle =0$ and
since the shape operator is symetric, then $\langle
S_{T}(E_{1}),E_{2}\rangle =$ $\langle S_{T}(E_{2}),E_{1}\rangle $. \ By
above calculations the Gaussian curvature $K_{T}$ is determined by%
\begin{eqnarray*}
K_{T} &=&\det S_{T} \\
&=&-(\langle S_{T}(E_{2}),E_{1}\rangle )^{2}
\end{eqnarray*}%
where $S_{T}(E_{2})=D_{E_{2}}N_{T}=\frac{1}{\left\Vert Y_{2}\right\Vert }.%
\frac{dN_{T}}{ds}$.

We obtain the Gaussian curvature by doing essential operations as%
\begin{eqnarray}
K_{T} &=&-\frac{\kappa ^{2}\kappa ^{\ast 2}}{(\kappa ^{2}(-\langle \beta
_{N},T\rangle +v)^{2}+\kappa ^{\ast 2})^{2}}.  \TCItag{1} \\
&&  \notag
\end{eqnarray}%
Also the mean curvature $H_{T}$ is given by%
\begin{eqnarray}
H_{T} &=&Tr(S_{T})  \notag \\
&=&\langle S_{T}(E_{2}),E_{2}\rangle  \notag \\
&=&\frac{1}{\left\Vert Y_{2}\right\Vert ^{3}}.\left\{ 
\begin{array}{c}
-\kappa ^{2}\tau \langle \beta _{N},T\rangle ^{2}+\langle \beta
_{N},T\rangle (\kappa ^{\ast \prime }\kappa +2v\kappa ^{2}\tau -\kappa
^{\ast }\kappa ^{\prime }) \\ 
-\kappa ^{\ast }\kappa \langle \beta _{N},T\rangle ^{\prime }-v\kappa \kappa
^{\ast \prime }-v^{2}\kappa ^{2}\tau -\kappa ^{\ast 2}\tau +v\kappa ^{\ast
}\kappa ^{\prime }%
\end{array}%
\right\}  \TCItag{2}
\end{eqnarray}%
We obtain Gaussian and mean curvatures of other ruled surfaces below by
doing similar calculations.

For ruled surface corressponded to dual spherical principal normal
representation curve;%
\begin{eqnarray}
K_{N} &=&-\frac{(\kappa \kappa ^{\ast }+\tau \tau ^{\ast })^{2}}{\left\Vert
Y_{2_{N}}\right\Vert ^{4}}  \TCItag{3} \\
&&  \notag
\end{eqnarray}%
where 
\begin{equation*}
\left\Vert Y_{2_{N}}\right\Vert ^{2}=%
\begin{array}{c}
\kappa ^{2}\langle \beta _{T},N\rangle ^{2}+(2\kappa \tau ^{\ast }-2\kappa
^{2}v)\langle \beta _{T},N\rangle +\kappa ^{\ast 2} \\ 
-(2\tau \kappa ^{\ast }+2\tau ^{2}v)\langle \beta _{B},N\rangle +2\kappa
^{\ast }\tau v+\tau ^{2}\langle \beta _{B},N\rangle ^{2} \\ 
+\tau ^{\ast 2}-2\kappa \tau ^{\ast }v+v^{2}(\kappa ^{2}+\tau ^{2})%
\end{array}%
\end{equation*}%
and%
\begin{eqnarray}
H_{N} &=&\frac{1}{\left\Vert Y_{2_{N}}\right\Vert ^{3}}.\left\{ 
\begin{array}{c}
(\kappa ^{\ast ^{\prime }}\kappa -\kappa ^{^{\prime }}\kappa ^{\ast }+\tau
^{\ast ^{\prime }}\tau -\tau ^{^{\prime }}\tau ^{\ast })(\langle \beta
_{T},N\rangle -v) \\ 
+(\kappa ^{\prime }\tau -\tau ^{\prime }\kappa )(\langle \beta _{T},N\rangle
^{2}-2v\langle \beta _{T},N\rangle +v^{2}) \\ 
-\langle \beta _{T},N\rangle ^{\prime }(\kappa \kappa ^{\ast }+\tau \tau
^{\ast })+\kappa ^{\ast ^{\prime }}\tau ^{\ast }-\tau ^{\ast \prime }\kappa
^{\ast }%
\end{array}%
\right\}  \TCItag{4} \\
&&  \notag
\end{eqnarray}%
are obtained. For ruled surface corressponded to dual spherical binormal
representation curve;%
\begin{eqnarray}
K_{B} &=&-\frac{\tau ^{2}\tau ^{\ast 2}}{\left\Vert Y_{2_{B}}\right\Vert ^{4}%
}  \TCItag{5} \\
&&  \notag
\end{eqnarray}%
where%
\begin{equation*}
\left\Vert Y_{2_{B}}\right\Vert ^{2}=\tau ^{2}(\langle \beta _{N},B\rangle
-v)^{2}+\tau ^{\ast 2}
\end{equation*}%
and%
\begin{eqnarray}
H_{B} &=&\frac{1}{\left\Vert Y_{2_{B}}\right\Vert ^{3}}\left\{ 
\begin{array}{c}
-\tau ^{2}\kappa \langle \beta _{N},B\rangle ^{2}+\langle \beta
_{N},B\rangle (-\tau ^{\prime }\tau ^{\ast }+2v\tau ^{2}\kappa +\tau ^{\ast
^{\prime }}\tau ) \\ 
-\tau \tau ^{\ast }\langle \beta _{N},B\rangle ^{\prime }-\kappa \tau ^{\ast
2}+v\tau ^{\prime }\tau ^{\ast }-v\tau ^{\ast ^{\prime }}\tau -v^{2}\tau
^{2}\kappa%
\end{array}%
\right\}  \TCItag{6} \\
&&  \notag
\end{eqnarray}
are obtained.

\section{Results}

In this section we will express theorems and use Gaussian and mean
curvatures to prove them.

\begin{theorem}
The non-developable ruled surface in $\mathbb{R}^{3}$ which is corressponded
to dual spherical representation curve $(T)$ is Weingarten and minimal
surface if and only if $\kappa ^{\ast }=0$ and $\tau =0$, for $\kappa \neq 0$
and $\tau ^{\ast }\neq 0$, where $\widehat{\kappa }=\kappa +\varepsilon
\kappa ^{\ast }$ and $\widehat{\tau }=\tau +\varepsilon \tau ^{\ast }$ are
curvature and torsion of the dual curve respectively.
\end{theorem}

\begin{proof}
In equations (1) and (2) if we put $\kappa ^{\ast }=0$ and $\tau =0$, for $%
\kappa \neq 0$ and $\tau ^{\ast }\neq 0$ we find that Gaussian curvature and
mean curvature are zero. Since Gaussian curvature is zero then the ruled
surface is developable,so it is Weingarten surface because of the theorem
given in \cite{Kuhnel} and the ruled surface is minimal due to the mean
curvature which is zero.

The opposite condition can be proved similarly.
\end{proof}

\begin{theorem}
The non-developable ruled surface in $\mathbb{R}^{3}$ which is corressponded
to dual spherical representation curve $(N)$ is Weingarten and minimal
surface if and only if $\kappa =0$ and $\tau ^{\ast }=0$ \ or $\kappa ^{\ast
}=0$ and $\tau =0$, where $\widehat{\kappa }=\kappa +\varepsilon \kappa
^{\ast }$ and $\widehat{\tau }=\tau +\varepsilon \tau ^{\ast }$ are
curvature and torsion of the dual curve respectively.
\end{theorem}

\begin{proof}
The proof is similar to previous proof of theorem, it can be done by taking
equations (3) and (4).
\end{proof}

\begin{theorem}
The non-developable ruled surface in $\mathbb{R}^{3}$ which is corressponded
to dual spherical representation curve $(B)$ is Weingarten and minimal
surface if and only if $\kappa =0$ and $\tau ^{\ast }=0$, for $\kappa ^{\ast
}\neq 0$ and $\tau \neq 0$, where $\widehat{\kappa }=\kappa +\varepsilon
\kappa ^{\ast }$ and $\widehat{\tau }=\tau +\varepsilon \tau ^{\ast }$ are
curvature and torsion of the dual curve respectively.
\end{theorem}

\begin{proof}
In this proof, equations (5) and (6) illuminate the solution.
\end{proof}

\bigskip

\.{I}lkay Arslan G\"{u}ven

University of Gaziantep

Department of Mathematics

\c{S}ehitkamil, 27310, Gaziantep, Turkey

E-mail: iarslan@gantep.edu.tr

\bigskip

Semra Kaya Nurkan, Murat Kemal Karacan

University of U\c{s}ak

Department of Mathematics

64200, U\c{s}ak, Turkey

E-mails: semrakaya\_gs@yahoo.com, mkkaracan@yahoo.com

\end{document}